\documentclass[12pt]{amsart}          %
\usepackage[a4paper]{geometry}	    
\usepackage{amsfonts,amsmath,latexsym,amssymb} 
\usepackage{xcolor}
\usepackage{enumitem, cite}
\usepackage[T2A]{fontenc}
\usepackage[utf8]{inputenc}
\usepackage[english]{babel}
\usepackage{appendix}

\newtheorem{theorem}{Theorem}

\newtheorem{lemma}{Lemma}

\newtheorem{definition}{Definition}

\newtheorem{remark}{Remark}

\newcommand{\RR}{\mathbb R}
\newcommand{\NN}{\mathbb N}
\newcommand{\ZZ}{\mathbb Z}
\newcommand{\EE}{\mathcal{E}}
\allowdisplaybreaks

\begin{document}

\title{Fixed Points Theorems in Hausdorff M-distance Spaces}

\author{Vladyslav Babenko}
\author{Vira Babenko}
\author{Oleg Kovalenko}
	
\begin{abstract}
We prove fixed point theorems in a space with a distance function that takes values in a partially ordered monoid. On the one hand, such an approach allows one to generalize some fixed point theorems in a broad class of spaces, including metric and uniform spaces. On the other hand, compared to the so-called cone metric spaces and $K$-metric spaces, we do not require that the distance function range has a linear structure. We also consider several applications of the obtained fixed point theorems. In particular, we consider the questions of the existence of solutions of the Fredholm integral equation in $L$-spaces.
\end{abstract}


\maketitle
	
\section{Introduction}
Fixed Point Theory is a well developed domain of Analysis and Topology (see books~\cite{Granas,Kirk,carl2014fixed} and references therein). It has a numerous amount of applications, in particular in  Numerical Analysis, the Theory of ODEs and PDEs, Integral Equations, Mathematical, Logic Programming, and others. 

One of the most well known fixed point theorems is the contraction mapping theorem in metric spaces, which goes back to Picard, Banach and Caccioppoli. The contraction mapping principle was generalized in various directions, see, e.g. mentioned above books and surveys~\cite{Zabrejko, Jancovic}.
 A majority of the contraction mapping theorems use the following scheme of the proof. Consider  a Picard sequence 
$
\{x_0,x_1=f(x_0),\ldots,x_n=f(x_{n-1}),\ldots\,\}\;
$
 and using the fact that  $f$ is contractive in some sense, establish that it is a Cauchy sequence. Using completeness of the space, obtain a point $x$, such that $x_n\to x$, as $n\to\infty.$ Continuity (in some sense) of $f$ gives $f(x)=x$. 

Such scheme can be (and in many cases was) realized in many spaces, which are more general than usual metric spaces. Among such spaces are various generalizations of metric spaces with co-domain $\RR_+$ (see~\cite{Hitzler,Jleli,Roldan,VillaMorales,JS2,Kucha} and references therein).

Spaces with metric that takes values in partially ordered vector spaces were introduced by Kurepa~\cite{Kurepa} in~1934. Later Kantorovich developed a theory of normed spaces with norms that take values in complete vector latices, see~\cite{Kantorovich}. These objects appeared to be fruitful in the study of functional equations using iterative methods, and in related questions of Analysis, see~\cite{KVP}. 
 Results and references on the fixed point theorems for the metric spaces with metric that takes values in a partially ordered vector spaces ($K$-metric spaces and cone-metric spaces) can be found in surveys~\cite{Zabrejko, Jancovic}. In particular, in these articles the question whether results on cone metric or $K$-metric spaces can be reduced to results on ordinary metric spaces, is discussed.
 Parametric metric spaces~\cite{Husseinatall}, probabilistic metric spaces~\cite{Schweizer,HadzicPap}, fuzzy metric spaces~\cite{Schweizer,Rakic} and intuitionistic metric spaces (see for references~\cite{Husseinatall, Yadav}), gauge spaces {(see~\cite[Chapter~6]{Kelley})} also were investigated in Fixed Point Theory.
 
A more general class of spaces is the class of distance spaces, or symmetries (or semi-metric spaces).
Some results regarding the fixed point theorems in real-valued distance spaces can be found in~\cite{Kirk, Choban, ChobanBerinde,Arandelovichatall,Suzuki} (see also references therein).
 Some results on fixed points in  cone distance spaces can be found in~\cite{Radenovic_Kadelburg,Zhang}.

Different definitions of the notion of contraction were considered in Fixed Point Theory (for survey and references see~\cite{Granas,Kirk,kirk_shahzad,Wong}), including non-linear, Meir-Keeler, Caristi, \'Ciri\'c, Caccioppolli contractions, and others.

We consider contraction mappings theorems in distance spaces with distances that take values in a partially ordered monoid, so that there might be no linear structure in the range of the distance function. We restrict ourselves to consideration of contractions of Meir--Keeler type, Caristi type and introduced in Section~\ref{s::sequentialContractions} sequential non-linear contractions of \'Ciri\'c and Caccioppoli types.

Note that each uniform space with separating axiom can be metrized with a metric that takes values in a partially ordered monoid, see  Section~\ref{s::examples}. {In view of this fact the result by Taylor~\cite{Taylor} about contraction mapping theorems in uniform spaces becomes a partial case of Theorem~\ref{th::M-KContraction}}. As it is well known, metrization of a uniform space with a real-valued metric is possible only for the uniform spaces with countable base. Questions of metrization of a uniform spaces with $K$-space valued metrics were considered in~\cite{Kusraev}.

The article is organized as follows. 
In Section~\ref{s::MDistanceSpaces} we adduce necessary definitions of a partially ordered monoid $M$, $M$-valued distance functions $d_{X,M}(x,y)$ in a set $X$, the notion of a series in $M$. In order to introduce convergence in an $M$-distance space, we define a family $Z(M)$ of null sequences in the monoid $M$ and say that $x_n\to x$ as $n\to\infty$ if $\{ d_{X,M}(x_n,x)\}\in Z(M)$. We also consider an example of a special family of null sequences $Z_{\EE}(M)$ and examples of $M$-distance spaces. {In order to prove fixed point theorems, we need the completeness of the space and the uniqueness of the limit of sequences. We give two generally speaking different definitions of completeness and briefly discuss their interconnection. We call Hausdorff the distance spaces that have the uniqueness of the limit property. We also discuss various conditions that guarantee Hausdorffness and provide examples of such spaces.}

In Section~\ref{s::contractionTheorems} 
we prove {Meir--Keeler type, Caristi type} fixed point theorems and fixed points theorems  for the mappings defined on Hausdorff $M$-distance spaces and partially ordered $M$-distance spaces that satisfy sequential contraction conditions of \'Ciri\'c and Caccioppoli types.

In Section~\ref{s::applications} we consider several applications of the obtained fixed point theorems. In particular, we consider the 
 questions of existence of solutions of the Fredholm integral equation in $L$-spaces (i.e. semi-linear metric spaces with two additional axioms, which connect the metric with the algebraic operations), see~\cite{Vahrameev, Babenko19, Babenko20, Babenko21}. We obtain significantly more general than in~\cite{Babenko19} conditions for existence and uniqueness of a solution of the  Fredholm integral equation. Observe that the proof of this result essentially uses a fixed point theorem applied to a distance space of continuous functions with a non-real valued metric, even though it has a usual real-valued metric. This example shows that consideration of various generalizations of metric spaces might give applications even in ordinary metric spaces.
In this section we also discuss possibilities to apply the obtained in Section~\ref{s::contractionTheorems} results in order to obtain contraction mapping theorems for random fixed points (see e.g.~\cite{Nietoatall}) and multiple fixed points (see e.g.~\cite{ChobanBerinde}).

\section{Partially ordered monoids. Null sequences. Cauchy series.}\label{s::MDistanceSpaces}
\begin{definition}
A set $M$ with associative binary operation $+$ is called a monoid, if there exists  $\theta_M\in M$ such that $\theta_M+x=x=x+\theta_M$ for all $x\in M$.

\end{definition}

\begin{definition}
A set $X$ is called partially ordered, if for some pairs of elements $x,y\in X$ the relation $x\le y$  is defined and is reflexive, transitive, and antisymmetric.
\end{definition}

\begin{definition}
{A monoid $M$ is called a partially ordered monoid, if it is a partially ordered set,} $M_+:=\{ x\in M\colon \theta_M\le x\}\neq \{\theta_M\},$ and the following condition holds: 
$$
    \text{If } x_1\leq y_1 \text{ and } x_2\leq y_2 \text{ then } x_1+x_2\leq y_1+y_2.
$$
\end{definition}

\begin{definition}
We say that the Riesz property holds in a partially ordered space $X$, if for each pair of elements $x,y\in X$  there exists an element $x\vee y\in X$ (which is called the supremum of $x$ and $y$) that satisfies the following properties:
\begin{enumerate}
    \item $x\leq x\vee y$ and $y\leq x\vee y$;
    \item If $z\in X$ is such that $x\leq z$ and $y\leq z$, then $x\vee y\leq z$.
\end{enumerate}
\end{definition}

\begin{definition}
For a partially ordered monoid $M$ denote by $Z(M)$ a family of sequences 
$
\{x_1,\ldots, x_n,\ldots\}\subset M_+
$
such that the following properties hold:
\begin{enumerate}
 \item $\{\theta_M,\ldots, \theta_M,\ldots\}\in Z(M)$ and if $x\in M_+$, $\{x,\ldots, x,\ldots\}\in Z(M)$, then $x=\theta_M$.
    \item If $\{x_n\}, \{y_n\}\in Z(M)$, then $\{x_n+y_n\}\in Z(M)$.
    \item If $\{x_n\}\in Z(M)$ and $\theta_M\leq y_n\leq x_n$ for all $n\in \NN$, then $\{y_n\}\in Z(M)$.
     \item Each subsequence $\{x_{n_k}\}$ of a sequence $\{x_n\}\in Z(M)$ belongs to $Z(M)$.
\end{enumerate}
$Z(M)$ will be called a family of null sequences in $M$.
\end{definition}

 As important example is the family $Z_{\EE}(M)$.
 \begin{definition}\label{def::solidM+}
    Assume that a partially ordered monoid $M$ contains a non-empty set $\mathcal{E}\subset  M_+\setminus \{ \theta_M \}$ that satisfies the following conditions:
    \begin{enumerate}
        \item If $x\in M_+$ is such that for all $\varepsilon\in \mathcal{E}$ one has $x\le \varepsilon$, then $x=\theta_M$.
    \item For each $\varepsilon\in \mathcal{E}$ there exists $\delta\in \mathcal{E}$ such that $\delta + \delta \leq \varepsilon$.
    \end{enumerate}
Define a family of null sequences $Z_{\EE}(M)$ as the family of sequences  $\{x_n\}\subset M_+$ with the following property:  
for any $\varepsilon\in \mathcal{E}$ there  exists $N \in \NN$ such that  $x_n < \varepsilon$ for all $n\geq N$.

\end{definition}


{
\begin{remark}\label{r::realNullSeqs}
The usual convergence to zero in the space $\RR_+$ is determined for example by the family $Z_{\EE}(\RR_+)$ with $\EE = (0,1)$.
\end{remark}
}

We also need the notion of a series in a partially ordered monoid. As usual, for a sequence $\{x_n\}\subset M_+$ the symbol 
\begin{equation}\label{seriesDef}
 \sum\nolimits\nolimits\nolimits_{k=1}^\infty x_k
\end{equation}
is called a series. The sum $s_n=\sum\nolimits\nolimits_{k=1}^nx_k$, $n\in \NN$, is called a partial sum of the series.

\begin{definition}\label{def::CauchySeries}
We say that series~\eqref{seriesDef} is a Cauchy series, 
if for all increasing sequences of natural numbers $\{n_k\}$ and $\{m_k\}$ such that $m_k\geq n_k$ for all $k\in\NN$, one has
$
\left\{\sum\nolimits\nolimits_{s=n_k}^{m_k} x_s\right\} \in Z(M).
$
By $\mathcal{CS}$ we denote the set of all sequences $\{x_n\}\subset M$ such that $\sum\nolimits\nolimits_{k=1}^\infty x_k$ is a Cauchy series.
\end{definition}
The notion of a Cauchy series is enough for our purposes. It is related to but different from the notion of a convergent series, which is well studied for normed linear spaces. Notice that the usual definition for convergence of a series requires a definition for convergence of an arbitrary sequence of elements. The introduced notion of a Cauchy series uses a family of null sequences in $M$, which in fact determines only convergence of sequences to $\theta_M$. Of course in the presence of a linear structure in $M$, convergence to an arbitrary element can be reduced to the convergence to $\theta_M$. In the case of a rather arbitrary set $M$, this is not the case.

\section{M-distance spaces. Hausdorffness. Completeness.}
\subsection{M-valued distances. Convergence. Hausdorffness }
\begin{definition}
Let $X$ be a set and $M$ be a partially ordered monoid. A mapping $d_{X,M}(\cdot,\cdot)\colon X\times X\to M_+$ such that $d_{X,M}(x,y) = d_{X,M}(y,x)$ for all $x,y\in X$ is called an $M$-valued
\begin{enumerate}
\item  Dislocated distance in $X$,  if  
$     d_{X,M}(x,y)=\theta_M \implies x=y;$
\item Distance in $X$,  if \hspace{1.5cm}
$
     d_{X,M}(x,y)=\theta_M\iff x=y;
$
\item Pseudo-distance in $X$, if \hspace{0.4cm}
$
    x= y \implies d_{X,M}(x,y)=\theta_M.
$
\end{enumerate}
 
The pair $(X,d_{X,M})$ is called a dislocated $M$-distance space, $M$-distance space, and $M$-pseudo-distance space respectively.
\end{definition}

In what follows we mostly deal with dislocated $M$-distance spaces and will call them $DM$-distance spaces for brevity.  A generalization of a classical example of dislocated metric spaces is obtained, if we set  $d_{M,M}(x,y)=x\vee y$, provided the Riesz property holds in $M$. We refer to~\cite{Hitzler} and references therein for a discussion of the history and motivation for investigations of dislocated spaces.

{Below we define some notions for dislocated $M$-distance spaces. Analogous notions can be defined in $M$-distance and $M$-pseudo-distance spaces; we do not do this explicitly.}

\begin{definition}
Let $(X, d_{X,M})$ be a $DM$-distance space, and let a family of null sequences $Z(M)$ in $M_+$ be chosen. 
The triple $(X, d_{X,M},Z(M))$ is called an equipped  $DM$-distance space ($EDM$-distance space for short).
\end{definition}

\begin{definition} 
Let $(X, d_{X,M},Z(M))$
be an $EDM$-distance space. We say that a sequence $\{ x_n\}\subset X$ converges to $x\in X$, and write $x_n\to x$ as $n\to \infty$, if  $\{d_{X,M}(x_n,x)\}\in Z(M)$.
\end{definition}
From the properties of the family of null sequences it follows that if $x_n\to x$ as $n\to \infty$, then each subsequence $\{x_{n_k}\}$ converges to $x$.

\begin{remark}
It is worth noting that different choices of the family of null sequences $Z(M)$ give generally speaking different $EDM$-distance spaces in the sense that depending of the family $Z(M)$, one obtains generally speaking different classes of converging  sequences. 
\end{remark}

{For example, two extremal situations occur in the case of degenerate families $Z(M)$ of null sequences. If $Z(M)$ contains only one sequence $\{\theta_M, \theta_M,\ldots\}$, then only sequences $\{x, x,\ldots,\}$ such that $d_{X,M}(x,x) = \theta_M$ are converging (to $x$). The other extremal situation occurs when $Z(M)$ contains all sequences of elements from $M_+$. In such situation each sequence in $X$ converges to any element $x\in X$.}

\begin{definition}
An $EDM$-distance space is called Hausdorff, if each converging sequence has a unique limit.
\end{definition}
Observe that  generally speaking one can't expect an $M$-pseudo-distance space to be Hausdorff (even if the distance function satisfies some kind of triangle inequality). However, if we have an enough rich family of pseudo-distances, then we can construct a Hausdorff $EDM$-distance space. On this path, we build a generalization of gauge spaces, see Example~\ref{ex::gaugeSpaces} in Section~\ref{s::examples}.

Conditions on distance functions that allow to preserve various properties of limit for reals were considered even before the notion of a metric space was introduced. We refer to articles~\cite{Frechet06,Hildebrandt,Pitcher,Wilson} which contain developments in this direction  (see also ~\cite{Arandelovic,Arandelovichatall} and references therein).
 Below we present several types of sufficient conditions for an $EDM$-distance space to be Hausdorff.

\subsection{Hausdorffness conditions}

\subsubsection{M-metric spaces and their generalizations}\label{s::zetaMetricSpace}

There are many different generalizations of the notion of a metric space in the case of real-valued metrics. Some references can be found in~\cite{Jleli,Roldan, VillaMorales,JS2,Kucha}. For $M$-valued metrics we adduce a generalization of a recent extension of the notion of a metric (see~\cite{JS2,Kucha}), which seems to be new even in the case $M=\RR$.

Let $\mathcal{M}$ be $M_+$ or $M_+\setminus \{\theta_M\}$.
Assume two functions 
$\phi\colon \mathcal{M}\to  \mathcal{M}\;\text{and}\; \zeta\colon \mathcal{M}^2\to \mathcal{M},$ 
satisfy the following conditions: if  $\{\alpha_n\},\{ \beta_n\}\subset \mathcal{M}$, then
\begin{equation}\label{phiIsNull}
    \{\phi(\alpha_n)\}\in Z(M)\implies \{\alpha_n\}\in Z(M)
\end{equation}
and 
\begin{equation}\label{zetaIsNullSq}
    \{\alpha_n\},\{\beta_n\}\in Z(M) \implies  \{\zeta (\alpha_n,\beta_n)\}\in Z(M).
\end{equation}
\begin{definition}
We say that an $EDM$-distance space is a 
$ED\mathcal{M}_{(\phi,\zeta)}$- metric space, if for arbitrary $x,y,z\in X$ such that $d_{X,M}(x,y),d_{X,M}(x,z),d_{X,M}(y,z)\in \mathcal{M}$ 
\begin{equation}\label{phi_psi}
\phi (d_{X,M}(x,y))\le \zeta (d_{X,M}(x,z),d_{X,M}(z,y)).    
\end{equation}
\end{definition}
\begin{lemma}
Each $ED\mathcal{M}_{(\phi,\zeta)}$- metric space is Hausdorff.
\end{lemma}
\begin{proof} We prove only the case $\mathcal{M}=M_+\setminus \{ \theta_M\}$; the second case is simpler.
Assume that a sequence $\{ x_n\}\subset X$ is such that $x_n\to x, x_n\to y$ as $n\to\infty$, and $x\neq y$. Two cases are possible:
\begin{enumerate}
    \item There exists  $n_0\in \NN$  such that $d_{X,M}(x_n,x)\neq \theta_M$ and $d_{X,M}(x_n,y)\neq \theta_M$ for all  $n\ge n_0$.
    \item There exists a subsequence   $\{ x_{n_k}\}$ such that for all $k$ one has $x_{n_k}=x$ (or $x_{n_k}=y$; for definiteness we assume that $x_{n_k}=x$.)
\end{enumerate}
In the first case for all  $n\in\NN$
$$
\phi(d_{X,M}(x,y))\le \zeta(d_{X,M}(x_{n_0+n},x),d_{X,M}(x_{n_0+n},y)),
$$
 hence $\{\phi(d_{X,M}(x,y))\}\in Z(M)$, which by~\eqref{phiIsNull} implies $\{d_{X,M}(x,y)\}\in Z(M)$ and $x = y$, and thus leads to a contradiction.

In the second case we obtain a constant sequence $\{ x_{n_k}\}=\{ x\}$ that converges to $y$. This again implies $\{d_{X,M}(x,y)\}\in Z(M)$ and $x = y$, which is impossible.
\end{proof}

An important example of a  $EDM_{(\phi,\zeta)}$-metric space is obtained, if in~\eqref{phi_psi}  $\phi$ is the identity mapping. Inequality~\eqref{phi_psi} becomes
$$
d_{X,M}(x,y)\le \zeta (d_{X,M}(x,z),d_{X,M}(z,y)).    
$$
The obtained space in the case $\mathcal{M} = M_+$ will be called an $EDM_\zeta$-metric space.
If additionally  $\psi(\alpha,\beta)=\alpha+\beta$, then we obtain an $EDM$-metric space, with the previous inequality becoming
$$
d_{X,M}(x,y)\le d_{X,M}(x,z)+d_{X,M}(z,y). $$

$K$-metric, $K$-normed, and Cone metric spaces
are partial cases of $EDM$-metric spaces.  Parametric metric spaces,
probabilistic metric spaces
fuzzy metric spaces
and intuitionistic metric spaces,
gauge spaces
can also  be viewed as $EDM$-metric spaces. {Example~\ref{ex::gaugeSpaces} in Section~\ref{s::examples} demonstrates this statement for generalized gauge spaces.}

\subsubsection{Fr\'{e}chet-Wilson type conditions}
Observe that   the definition of Hausdorffness can be rewritten in the following equivalent way: if $\{x_n\}\subset X$ and $a,b\in X$, then
$$
\{d_{X,M}(x_n,a)\},\{d_{X,M}(x_n,b)\}\in Z(M)\implies \{d_{X,M}(a,b)\}\in Z(M).
$$

Consider the following generalizations of the Fr\'{e}chet-Wilson properties~\cite{Wilson}.
\begin{definition}\label{def::FWProps}
An $EDM$-distance space $(X,d_{X,M},Z(M))$ is said to satisfy the 
\begin{enumerate}
    \item Weak Fr\'{e}chet-Wilson property, if for any sequences $\{x_n\},\{y_n\}\subset X$, $z\in X$,
    $$
\{d_{X,M}(x_n,y_{n})\},\{d_{X,M}(y_{n},z)\}\in Z(M)\implies\{d_{X,M}(x_n,z)\}\in Z(M).
$$ 
\item Fr\'{e}chet-Wilson property, if for any sequences $\{ x_n\},\{ y_n\},\{ z_n\}\subset X$, 
$$
\{d_{X,M}(x_n,z_{n})\},\{d_{X,M}(z_{n},y_n)\}\in Z(M)\implies\{d_{X,M}(x_n,y_{n})\}\in Z(M).
$$
\item Strong Fr\'{e}chet-Wilson, if  for arbitrary sequence $\{x_n\}\subset X$,  and increasing sequences of natural numbers $\{n_k\}$ and $\{m_k\}$ such that $m_k> n_k$ for all $k\in\NN$,
$$
\left\{\sum\nolimits\nolimits_{s = n_k}^{m_k-1} d_{X,M}(x_s,x_{s+1})\right\}\in Z(M)\implies\left\{ d_{X,M}(x_{n_k},x_{m_k})\right\}\in Z(M).
$$
\end{enumerate}
\end{definition}

It is easy to see that the strong Fr\'{e}chet-Wilson property implies the Fr\'{e}chet-Wilson property, which implies the weak Fr\'{e}chet-Wilson property, which in turn implies Husdorffness of an $M$-distance space. Hausdorffness and the first two of the above properties were considered {in the case $M= \RR$ }in~\cite{Wilson} as different ways to relax the triangle inequality. It is not hard to construct examples that show that generally speaking no two of the three properties are equivalent. Moreover, it is easy to see that the strong Fr\'{e}chet-Wilson property is still weaker than the triangle inequality (in the sense that each $EDM$-metric space satisfies the strong Fr\'{e}chet-Wilson condition), and as Example~\ref{ex::strongFWIsNotEquivalent} from Section~\ref{s::examples} shows, it is not equivalent to  the Fr\'{e}chet-Wilson property.

\subsubsection{Hausdorffness conditions in terms of semi-continuity of the distance function}
\begin{definition}
Let an $EDM$-distance space $(X,d_{X,M}, Z(M))$ and $x,y\in X$ be given. We say that a  function $f\colon X\to M$ is semi-continuous from below at a point $x\in X$, if for each sequence $\{ x_n\}\subset X$ that converges to $x$ as $n\to\infty$, there exists a subsequence $\{ x_{n_k}\}$ and a null sequence $\{\alpha_k\}\in Z(M)$ such that for all $k\in\NN$
\begin{equation}\label{semicontinuousDistance}
  f(x)\le f(x_{n_k})+\alpha_k.
\end{equation}
\end{definition}
\begin{lemma}
$EDM$-distance space $(X,d_{X,M}, Z(M))$ is Hausdorff, provided for each $y\in X$ the function  $d_{X,M}(\cdot,y)$ is semi-continuous from below.
\end{lemma}
\begin{proof}
Let $x_n\to x$ and $x_n\to y$ as $n\to\infty$. Due to~\eqref{semicontinuousDistance} we obtain 
\begin{equation}\label{distanceSemiContinuity}
  d_{X,M}(x,y)\le d_{X,M}(x_{n_k},y)+\alpha_k,\,k\in\NN,
\end{equation}
  where $\{\alpha_k\}\in Z(M)$. 
Since $x_n\to y$, we obtain $\{ d_{X,M}(x_{n_k},y)\}\in Z(M)$. Thus $\{ d_{X,M}(x,y)\}\in Z(M)$,  $d_{X,M}(x,y)=\theta_M$ and $x=y$.
\end{proof}
Observe that semi-continuity of $d_{X,M}(\cdot ,y)$ (see inequality~\eqref{distanceSemiContinuity}) can be substituted by a more general inequality
\begin{equation}\label{zetaSemicontinuousDistance}
 d_{X,M}(x,y)\le\zeta( d_{X,M}(x_{n_k},y),\alpha_k),\,k\in\NN,
\end{equation}
where $\zeta$ satisfies~\eqref{zetaIsNullSq}. Moreover, instead of one $\zeta$ for all pairs $x,y\in X$, one can allow  $\zeta $ to depend on $x$ and $y$. 

In the case, when $M= \RR_+$, $Z(M)$ determines the usual convergence, and $\zeta(\alpha,\beta) = \xi(\alpha) + \beta$ with  monotone continuous at $0$ function $\xi$ such that $\xi(0) = 0$, condition~\eqref{zetaSemicontinuousDistance} can be rewritten in the following equivalent and more convenient form
$$
d_{X,M}(x,y) \leq \limsup\nolimits_{n\to\infty}\xi(d_{X,M}(x_{n_k},y)).
$$
In such form this condition appeared in~\cite{Jleli,Roldan,VillaMorales}.

\subsection{Cauchy sequences. Completeness}

\begin{definition}
A sequence $\{ x_n\}\subset X$ is called a Cauchy sequence, if for each increasing sequences of natural numbers $\{n_k\}$ and $\{m_k\}$ such that $m_k> n_k$ for all $k\in\NN$,  one has $\left\{d_{X,M}(x_{n_k},x_{m_k})\right\}\in Z(M).$ We denote by $\mathcal{CH}$ the family of all Cauchy sequences.
\end{definition}
The Fr\'{e}chet-Wilson property implies that each converging sequence $\{x_n\}$ belongs to $\mathcal{CH}$. In~\cite{Wilson} an example of a convergent but not a Cauchy sequence in a space (with $M = \RR$) that satisfies the weak Fr\'{e}chet-Wilson property  was constructed.
\begin{definition}
 We denote by $\mathcal{CW}$ the family of all sequences $\{x_n\}\subset X$ such that $$\{d_{X,M}(x_n,x_{n+1})\}\in \mathcal{CS}.$$
\end{definition}
The following statement is obvious.
\begin{lemma}\label{l::connection}
In each $EDM$-distance space that satisfies the strong Fr\'{e}chet-Wilson property, one has  $\mathcal{CW}\subset \mathcal{CH}$.
\end{lemma}
At the same time this relation is not true in the general case, as shown by the following example. 
Let $X=\NN\cup\Omega\cup \{\infty\}$, where $\Omega=\{ \omega_1,\omega_2,\ldots \}$ is a countable set of different elements. Let also $M=\RR_+$, {and $Z(\RR_+)$ is the family of all non-negative sequences converging to $0$ in the usual sense.} Let $d_{X,\RR_+}(x,y)=0$ if $x=y$, and 
$$d_{X,\RR_+}(x,y)=d_{X,\RR_+}(y,x)=
\begin{cases}
1, & \text{if } x,y\in\NN, \text{ or } x,y\in \Omega, x\neq y,\\
\frac 1{n^2}, & \text{if $x=n\in \NN$ and $y\in \Omega\cup\{\infty\}$},\\
 \frac 1{j^2}, &\text{if $x=\omega_j\in \Omega$ and $y=\infty$}.
\end{cases}
$$
Consider the sequence $\{x_k\}=\{ 1,\omega_1, 2,\omega_2,\ldots , n,\omega_n\ldots\}$. It is easy to see that this sequence satisfies the following properties. The series $
\sum\nolimits\nolimits_{k=1}^\infty d_{X,\RR_+}(x_k,x_{k+1})
$ converges (in the usual sense for real series), and hence it is a Cauchy series in the sense of Definition~\ref{def::CauchySeries}. $\{x_k\}$ is not a Cauchy sequence, since $d_{X,\RR_+}(n,m)=1$ for all  $n,m\in \NN$, $n\neq m$. It converges to $\infty$.

 \begin{definition}
We call an $EDM$-distance space $(X,d_{X,M}, Z(M))$ $\mathcal{H}$-complete ($\mathcal{W}$-complete), if each sequence $\{x_n\}\in \mathcal{CH}$ (resp. $\{x_n\}\in \mathcal{CW}$)  converges to an element from the space $X$.
\end{definition}
Observe that Lemma~\ref{l::connection} implies that each $\mathcal{H}$-complete $EDM$-distance space that satisfies the strong Fr\'{e}chet-Wilson property is $\mathcal{W}$-complete.

\subsection{Convergent and Cauchy sequences in $(X,d_{X,M},Z_{\EE})$}

If the family of null sequences is chosen to be $Z_{\EE}(M)$, then some notions introduced in previous subsections can be written in a simpler way. The following lemma is obvious. 
\begin{lemma}\label{l::econvergence}
If $Z(M) = Z_{\EE}(M)$, then for a sequence $\{ x_n\}\subset X$ and $x\in X$ one has that $x_n\to x$ {as} $n\to\infty$ if and only if for any $\varepsilon\in \EE$ there exists $N\in\NN$ such that for $n\geq N$ one has $d_{X,M}(x_n,x)< \varepsilon$.
\end{lemma}

\begin{lemma}\label{l::eCauchySeries}
If $Z(M) = Z_{\EE}(M)$, then a series $\sum\nolimits\nolimits_{n=1}^\infty x_n$, $\{x_n\}\subset M_+$, is a Cauchy series if and only if for all $\varepsilon\in \EE$ there exists $N\in\NN$ such that for all $m\geq n\geq N$ one has
\begin{equation}\label{eFundamelity}
    \sum\nolimits\nolimits_{k=n}^m x_k < \varepsilon.
\end{equation}
\end{lemma}
\begin{proof}
Assume that $\sum\nolimits\nolimits_{n=1}^\infty x_n$ is a Cauchy series but there exists $\varepsilon\in\EE$ such that for each $N\in\NN$ there exist $m\geq n\geq N$ such that   inequality~\eqref{eFundamelity} does not hold. Then one can build increasing sequences of natural numbers $\{m_k\}$ and $\{n_k\}$, $n_k\leq m_k$ for all $k\in \NN$ such that the inequality  $ \sum\nolimits\nolimits_{s=n_k}^{m_k} x_s < \varepsilon$ does not hold for all $k\in\NN$. However, this contradicts to the assumption 
\begin{equation}\label{generalFundamelity}
 \left\{\sum\nolimits\nolimits_{s=n_k}^{m_k} x_s\right\}\in Z_{\EE}(M).   
\end{equation}

If inequality~\eqref{eFundamelity} holds, then for arbitrary $\varepsilon\in \EE$ and arbitrary increasing sequences of natural numbers $\{m_k\}$, $\{n_k\}$ one has $\sum\nolimits\nolimits_{s=n_k}^{m_k} x_s<\varepsilon$ whenever $k$ is such that  $n_k>N$, i.e.~\eqref{generalFundamelity} holds.
 \end{proof}
The proof of the following lemma is similar to the proof of Lemma~\ref{l::eCauchySeries}.
 \begin{lemma}\label{l::eCauchySequence}
 A sequence $\{x_n\}$ in an $EDM$-distance space $(X, M, Z_{\EE}(M))$ is a Cauchy sequence if and only if for all $\varepsilon\in\EE$ there exists $N\in\NN$ such that $d_{X,M}(x_n,x_m)<\varepsilon$ for all $m\geq n\geq N$.
 \end{lemma}
In the case $M= \RR$, related to the following lemma results for the Fr\'{e}chet-Wilson and weak Fr\'{e}chet-Wilson properties can be found in~\cite{Wilson}. We do not use this result, so omit its proof, which can be found in the preprint~\cite{fmDistSpaces}.
 \begin{lemma}
 For an $EDM$-distance space $(X, M, Z_{\EE}(M))$ to satisfy the strong Fr\'{e}chet-Wilson property it is sufficient, and  if there exists a null sequence $\{\varepsilon_n\}\subset\EE$,
 then it is necessary that for any $\varepsilon\in \mathcal{E}$ there exists $\delta\in\mathcal{E}$ such that for all $n\in\NN$ and $x_1,\ldots, x_n\in X$ 
$$
\sum\nolimits\nolimits_{k=1}^{n-1} d_{X,M}(x_k,x_{k+1})< \delta \implies d_{X,M}(x_1,x_n)< \varepsilon.
$$
 \end{lemma}

\subsection{Examples of Hausdorff distance spaces}\label{s::examples}

\begin{enumerate}[wide]
\item\label{ex::strongFWIsNotEquivalent} {\bf $M$-distance spaces that satisfy the strong Fr\'{e}chet-Wilson property but are not metric spaces.}
The triple $(\mathbb{R}, d_{\mathbb{R},\mathbb{R}}, Z_{\EE}(\RR_+))$, where
$$
    d_{\RR,\RR}(x,y) =
    \begin{cases}
     |x-y|,& |x-y|\leq 1,\\
     (x-y)^2, & |x-y|>1
    \end{cases}
$$
{and $Z_{\EE}(\RR_+)$ determines the usual convergence (see Remark~\ref{r::realNullSeqs})}
is an example of an equipped $M$-distance space that satisfies the strong Fr\'{e}chet-Wilson property and is not a metric space.

The triple $(\mathbb{R}, (x,y)\mapsto (x-y)^2, Z_{\EE}(\RR_+))$, 
is an example of an equipped $M$-distance space that satisfies the  Fr\'{e}chet-Wilson property, but does not satisfy the strong Fr\'{e}chet-Wilson property.

 Let $C(T ,X)$ be the space of continuous functions defined on a metric compact $T$ with values in an $L$-space $X$ with a metric $h_X(x, y)$.
Definitions and necessary facts from the theory of $L$-spaces, in particular the definition of the Lebesgue integral for an $L$-spaces valued functions that will be needed in Section~\ref{s::fredholmEquation},  can be found in~\cite{Vahrameev, Babenko20}.

Set  
$d(x,y)=d_{\RR,\RR}(h_X(x(\cdot), y(\cdot)), 0)$. We obtain a  $C(T,\mathbb{R})$-valued distance in $C(T ,X)$. With such distance function {and a family $Z(C(T,X))$ of point-wise convergent to zero sequences of continuous functions, the triple $(C(T ,X), d, Z(C(T,X)))$} becomes an equipped $M$-distance space that satisfies the strong Fr\'{e}chet-Wilson property.
    
 \item\label{ex::CartesianProducts}   {\bf Cartesian product of a finite number of $EDM$-distance spaces.} 
 Assume $(X,d_{X,M}, Z(M))$ is an $EDM$-distance space that satisfies one of the  Fr\'{e}chet-Wilson properties from Definition~\ref{def::FWProps}. 
 On $X^m$, $m\in\NN$, one can define different  $M$-distances to make $X^m$ an $EDM$-distance space that satisfies the same Fr\'{e}chet-Wilson property. For example, one can use the same family of null sequences $Z(M)$ and consider the distance function
$$
d_{X^m,M}^\Sigma((x_1,\ldots,x_m),(y_1,\ldots,y_m))=\sum\nolimits\nolimits_{k=1}^md_{X,M}(x_k,y_k),
$$
or, if $M$ satisfies the Riesz property,
$$
d_{X^m,M}^\vee((x_1,\ldots,x_m),(y_1,\ldots,y_m))=\bigvee_{k=1}^md_{X,M}(x_k,y_k).
$$
    
Another distance function can be defined as
$$
    d^m_{X^m,M^m}((x_1,\ldots,x_m),(y_1,\ldots,y_m))
    =
    (d_{X,M}(x_1,y_1),\ldots,d_{X,M}(x_m,y_m)).
$$
Here $M^m$ is a monoid with coordinate-wise addition and order,  and $$Z(M^m) = \left\{\{(x_1^n,x_2^n,\ldots, x_m^n)\}\colon \{x_1^n\},\ldots, \{x_m^n\}\in Z(M)\right\}.$$
    
\item {\bf Cartesian product of  an arbitrary family of $M$-pseudo-distance spaces.}\label{ex::infiniteProduct} Assume that an arbitrary set $A$ of parameters, and a family $\{ (X_\alpha, d_{X_\alpha,M_\alpha}, Z(M_\alpha))\}_{\alpha\in A}$ of equipped $M_\alpha$-pseudo-distance spaces are given. Let $M =  \prod_{\alpha\in A} M_\alpha$. In the product space $X = \prod_{\alpha\in A} X_\alpha$ one can define an $M$-valued pseudo-distance function  (it is clear that $M$ is a partially ordered monoid with coordinate-wise addition and partial order, and the zero element is the identity zero function of the variable $\alpha$) $d_{X,M}$ as follows: for $x,y\in X$, and $\alpha\in A$, 
$$
d_{X,M}(x,y)(\alpha) := d_{X_\alpha,M_\alpha}(x_\alpha, y_\alpha).
$$
The family 
$$Z(M) = \{\{x^n\}\subset M\colon \{x^n_\alpha\}\in Z(M_\alpha)\text{ for all } \alpha\in A\}$$ defines coordinate-wise convergence in $X$.  It is clear that the triple $(X,d_{X,M},Z(M))$ becomes an equipped  $M$-pseudo-distance space. 

 If each of the spaces $ (X_\alpha, d_{X_\alpha,M_\alpha},Z(M_\alpha))$ is $M_\alpha$-distance space ($M_\alpha$-metric space, Hausdorff distance space), then their Cartesian product $(X,d_{X,M},Z(M))$ is an $M$-distance space (resp. $M$-metric space, Hausdorff distance space).

\item {\bf Generalized gauge spaces.}\label{ex::gaugeSpaces}
 Let a family $\{(Y,d_{Y,M_\alpha}, Z(M_\alpha))\}_{\alpha\in A}$ of equipped  $M_\alpha$-pseudo-distance spaces be given. As in the previous example, we set $M=\prod_{\alpha\in\mathcal{A}}M_\alpha$;  let also $X_\alpha =Y$ for each $\alpha\in A$. The obtained Cartesian product $(Y^A, d_{Y^A,M},Z(M))$ becomes an equipped $M$-pseudo-distance space. Identifying the elements of the space  $Y$ with the constant mappings from $A$ to $Y$ and setting for  $x,y\in Y$
 $$
d_{Y,M}(x,y)(\alpha) := d_{Y,M_\alpha}(x, y),
 $$
 we obtain an equipped $M$-pseudo-distance space $(Y, d_{Y,M},Z(M))$. If the family of $M$-pseudo-metrics $d_{Y,M_\alpha}$ is enough rich in the sense that   
 \begin{equation}\label{familyOfPseudoMetricsIsRich}
   \text {for all } x,y\in Y,  x\neq y,\text{ there exists } \alpha\in A\text{ such that } d_{Y,M_\alpha}(x,y)\neq \theta_M,
  \end{equation}
 then the obtained space is in fact an equipped $M$-distance space. We call this space a generalized  gauge space.
 
 {In general, the obtained space might be non-Hausdorff, but if all $d_{Y,M_\alpha}$ are $M$-pseudo-metrics and~\eqref{familyOfPseudoMetricsIsRich} holds, then the obtained space is $M$-metric and hence a Hausdorff space.

In the case, when $M_\alpha =\RR_+$ for all $\alpha\in A$, we obtain the classical definition of a gauge space. The family of sets
$
V_{\alpha,r} = \{(x,y)\in Y\times Y\colon d_{Y,M_\alpha}(x,y)<r\}, \alpha\in A, r>0, $
is a subbase of entourages of the uniformity generated by the family $\{d_{Y,M_\alpha}\}$, see e.g.~\cite[Chapter~6]{Kelley}.  The convergence in  $Y$ in the topology of the uniformity is the same as the one in the equipped $M$-pseudo-distance space $(Y, d_{Y,M},Z(M))$. 
 }
    
\item {\bf Uniform spaces as $M$-metric spaces.}\label{ex::uniformSpace}
{
The previous example shows that one can determine the convergence in a uniform space generated by a family of pseudo-metrics using an equipped $M$-pseudo-distance space. Observe also that there is a natural linear structure in the range of the distance function of the corresponding $M$-pseudo-distance space. However, if the uniform structure is defined using a uniformity rather than a gauge, then it might be desirable to determine the convergence in terms of the given entourages of the uniformity rather than via implicit information about the gauges. In this example we construct an $M$-metric space with the same convergence in terms of entourages of the uniformity with separating axiom; if the Hausdorffness of the corresponding to the uniformity topology is not assumed, then the same construction will result in an $M$-pseudo-metric space.
}

We use the terminology from~\cite[Chapter~6]{Kelley}, see also~\cite[Chapter~8]{Engelking}. 
  Let $X$ be a set and $\Phi$ be a separating uniform structure on it.
  Denote by $M$ the set of all 
  subsets of $X\times X$ that contain the set $\Delta(X):=\{(x,x)\colon x\in X\}$. It becomes a partially ordered monoid, if we define the order by inclusion, set  $\theta_M=\Delta(X)$, and $A + B = A\circ B$, where as usually
$$
       A\circ B:=\{(x,y)\in X\times X\colon \text{there exists } z\in X \text{ such that } 
       (x,z)\in A \text{ and } (z,y)\in B\}.
$$
    Obviously, $M_+ = M$.
    Let $\EE$ be a symmetric base of entourages of the uniformity $\Phi$ and set 
  $$
    d_{X,M}(x,y) = \bigcap\left\{\varepsilon\in \EE\colon (x,y)\in \varepsilon\right\}.
  $$
  The family $\EE$ satisfies all properties of Definition~\ref{def::solidM+}, and hence we can consider the family $Z_{\EE}(M)$ of null sequences.  Next we show that $(X, d_{X,M},  Z_{\EE}(M))$ is an equipped $M$-metric space.
  
 If $d_{X,M}(x,y) = \Delta(X)$, then $(x,y)\in \Delta(X)$ and hence $x = y$.  Moreover, if $x = y$, then $(x,y) \in \varepsilon$ for all $\varepsilon\in \EE$ and hence $d_{X,M}(x,y) = \Delta(X)$ due to the separating axiom.   From the definition of the function $d_{X,M}$ it follows that $d_{X,M}(x,y) = d_{X,M}(y,x)$, due to symmetricity of entourages from $\EE$.   Finally, let $x,y,z\in X$. Then $(x,z)\in d_{X,M}(x,y)\circ d_{X,M}(y,z)$, and hence $d_{X,M}(x,z)\leq d_{X,M}(x,y) + d_{X,M}(y,z).$
  
 It is easy to see that  convergence in the space $(X,d_{X,M}, Z_{\EE}(M))$ coincides with the convergence in the topology induced by the uniform structure.
\end{enumerate}

 \section{ Contraction mapping theorems}\label{s::contractionTheorems}

\subsection{A Meir-Keeler-Taylor type  theorem}
In this section we show that consideration of spaces with monoid-valued distance can allow to unify the arguments for proofs of fixed point theorems in various seemingly different settings.

Below is  a result for Meir-Keeler~\cite{Meir69} type contractive mappings {in $EDM_\zeta$-metric spaces (see Section~\ref{s::zetaMetricSpace})}. In view of Example~\ref{ex::uniformSpace} of Section~\ref{s::examples} it also contains a  result by Taylor~\cite{Taylor} for contractive mappings in uniform spaces.

\begin{theorem}\label{th::M-KContraction}
Let $(X,d_{X,M},Z_{\EE}(M))$ be an $\mathcal{H}$-complete $EDM_\zeta$-metric space with coordinate-wise non-decreasing $\zeta$ that satisfies property~\eqref{zetaIsNullSq} and 
$
\zeta(\varepsilon,\delta)\geq \varepsilon,\, \zeta(\varepsilon,\delta)\geq \delta \text{ for all } \varepsilon, \delta\in \EE.
$
Assume $f\colon X\to X$ satisfies the following property: for each $\varepsilon\in \EE$ there exists $\delta\in \EE$ such that for all $x,y\in X$ 
\begin{equation}\label{taylor2}          d_{X,M}(x,y)\leq \zeta( \delta, \varepsilon) \implies d_{X,M}(f(x),f(y))<\varepsilon.
\end{equation}
Then $f$ has a fixed point, provided one of the two sets of properties holds:
\begin{enumerate}[label = \Roman*]
\item\label{conditionsOnM}
\begin{enumerate}
    \item\label{taylorUniqueness1} If $x,y\in X$ and $\alpha, \beta\in M_+$ are such that $d_{X,M}(x,y)< \alpha+\beta$, then there exists $z\in X$ such that $d_{X,M}(x,z) < \alpha$ and $d_{X,M}(z,y) < \beta$;
    \item\label{taylorUniqueness2} For arbitrary $\varepsilon\in \EE$ and $x,y\in X$ there exists $n\in\NN$ such that $d_{X,M}(x,y)< n\cdot \varepsilon$;
\end{enumerate}
\item \label{conditionsOnF}
\begin{enumerate}
    \item  There exists $x\in X$ such that 
\begin{equation}\label{taylor1}
  \{  d_{X,M}(f^n(x),f^{n+1}(x))\}\in Z_{\EE}(M).
\end{equation} 
\end{enumerate}
\end{enumerate}
If conditions~\ref{conditionsOnM} hold, then the fixed point is unique.
\end{theorem}

\begin{proof}
First of all observe that $f$ is non-expansive in the following sense: for arbitrary $\varepsilon\in\EE$,
\begin{equation}\label{fIsNonExpanding}
    d_{X,M}(x,y)< \varepsilon\implies d_{X,M}(f(x),f(y))<\varepsilon.
\end{equation}
Indeed, let $\delta\in \EE$ be chosen according to condition~\eqref{taylor2}. Since  $d_{X,M}(x,y)<\varepsilon \leq \zeta(\delta,\varepsilon  )$, we obtain $d_{X,M}(f(x),f(y))<\varepsilon$ as required. 

Next we show that if conditions~\ref{conditionsOnM} hold, then for arbitrary $a,b\in X$, 
\begin{equation}\label{taylor0}
  \{  d_{X,M}(f^n(a),f^{n}(b))\}\in Z_{\EE}(M).
\end{equation}

Let $\varepsilon\in\EE$ be arbitrary, $\delta\in\EE$ be chosen according to~\eqref{taylor2} and $n$ be chosen according to property~\eqref{taylorUniqueness2}, so that $d_{X,M}(a,b)<n\cdot \delta$. Due to Lemma~\ref{l::econvergence}, inclusion~\eqref{taylor0} is implied by \begin{equation}\label{taylorUniquenessOfFixedPoint}
d_{X,M}(f^m(a), f^m(b))< \varepsilon\text{ for all } m\geq n.
\end{equation}
In order to prove~\eqref{taylorUniquenessOfFixedPoint}, it is sufficient to show that 
\begin{equation}\label{inductionStatement}
d_{X,M}(a,b)<n\cdot \delta\implies d_{X,M}(f^n(a), f^n(b))< \varepsilon
\end{equation}
in virtue of~\eqref{fIsNonExpanding}. We proceed by induction on $n\in\NN$. If $n=1$, then 
$
d_{X,M}(a,b) < \delta \leq \zeta(\delta, \varepsilon),
$
which implies~\eqref{inductionStatement} with $n = 1$, due to~\eqref{taylor2}.
Assume that it holds for some $n\in\NN$ and 
$$
d_{X,M}(a,b)< (n+1)\delta = \delta + n\delta.
$$
In view of property~\eqref{taylorUniqueness1}, there exists $c\in X$ such that $d_{X,M}(a,c)<  \delta$ and $d_{X,M}(c,b) < n\delta$. In virtue of the inductive assumption, $d_{X,M}(f^n(c), f^n(b))< \varepsilon$, and  inequality~\eqref{fIsNonExpanding} implies that  $d_{X,M} (f^n(a),f^n(c))< \delta$. Applying the triangle inequality, $
d_{X,M} (f^n(a),f^n(b))
\leq \zeta( \delta, \varepsilon)
$
and thus by~\eqref{taylor2}, $d_{X,M} (f^{n+1}(a),f^{n+1}(b))< \varepsilon$, as desired.

Applying~\eqref{taylor0} to $a = x$ and $b = f(x)$ with arbitrary $x\in X$, we obtain~\eqref{taylor1}. Hence in both cases when either conditions~\ref{conditionsOnM} or~\ref{conditionsOnF} hold, there is $x\in X$ such that~\eqref{taylor1} holds.

Next we show that  $\{f^n(x)\}$ is a Cauchy sequence.  For a fixed $\varepsilon\in \EE$ choose $\delta\in \EE$ according to condition~\eqref{taylor2}. Due to~\eqref{taylor1} and Lemma~\ref{l::econvergence}, there exists  $n_0$ such that for all $n\ge n_0$ one has
\begin{equation}\label{taylor1Conseq}
d_{X,M}(f^n(x),f^{n+1}(x))<\delta \text{ and } d_{X,M}(f^n(x),f^{n+1}(x))<\varepsilon.  
\end{equation}
Let $n> n_0$ be fixed. We show by induction on $k$ that for all $k\in \NN$  one has
\begin{equation}\label{fundamentalityCondition}
d_{X,M}(f^n(x),f^{n+k}(x))<\varepsilon.
\end{equation}
For $k=1$ inequality~\eqref{fundamentalityCondition} holds due to~\eqref{taylor1Conseq}. Assume that it is true for some $k \geq 1$. Then due to~\eqref{taylor1Conseq}, $d_{X,M}(f^{n-1}(x),f^{n}(x))<\delta$, and $d_{X,M}(f^n(x),f^{n+k}(x))<\varepsilon$ by the inductive assumption. Due to the triangle inequality, $d_{X,M}(f^{n-1}(x),f^{n+k}(x))\leq \zeta(\delta,\varepsilon)$, and hence in virtue of condition~\eqref{taylor2}, 
$
d_{X,M}(f^n(x),f^{n+k+1}(x))<\varepsilon,
$
which completes the induction step. Hence $\{f^n(x)\}$ is a Cauchy sequence due to Lemma~\ref{l::eCauchySequence}.

Completeness of the space implies existence of a point $\overline{x}\in X$ such that 
\begin{equation}\label{orbitLimit}
    f^n(x)\to\overline{x} \text{ as } n\to\infty.
\end{equation}
Property~\eqref{fIsNonExpanding} implies that $f$ is continuous, i.e. if $x_n\to x$, then $f(x_n)\to f(x)$. 
From~\eqref{orbitLimit} and continuity of $f$ we obtain $f(f^n(x))\to f(\overline{x})$. Uniqueness of the limit implies
$f(\overline{x})=\overline{x}$.

Finally, we prove the uniqueness of the fixed point  in the case, when conditions~\ref{conditionsOnM} hold. Let $a,b$ be two fixed points of $f$. Then for all $m\in\NN$,  $d_{X,M}(a,b) = d_{X,M}(f^m(a),f^m(b))$ and hence~\eqref{taylor0} implies that $\{d_{X,M}(a,b)\}\in Z(M)$, thus $a=b$, as desired.
\end{proof}

\subsection{A Caristi type theorem}
In this section we present a Caristi type theorem (see e.g.~\cite{kirk_shahzad} and references therein) for $M$-distance spaces.
\begin{definition}
Let $M$ be a partially ordered monoid and $Z(M)$ be a family of null sequences  in $M$. The pair $(M,Z(M))$ is said to satisfy the Weierstrass property, if for arbitrary sequence $\{\alpha_n\}\subset M$ such that the sequence of partial sums $\left\{\sum\nolimits\nolimits_{k=1}^n \alpha_k\right\}$ is bounded, one has $\{\alpha_n\}\in \mathcal{CS}$.
\end{definition}

Observe that the monoids built in Examples~\ref{ex::infiniteProduct} and~\ref{ex::gaugeSpaces} of Section~\ref{s::examples} satisfy the Weierstrass property provided each component of the product $(M_\alpha, Z(M_\alpha))$ satisfies it. In particular, the Weierstrass property holds in the case $M = \RR_+$ with usual convergence to $0$. On the other hand, the monoid from Example~\ref{ex::uniformSpace} does not generally speaking satisfy the Weierstrass property, since $X\times X\in M_+$ and hence partial sums of any series are bounded.

\begin{definition}
We say that a function $f$ is weakly orbitally continuous, if 
\begin{equation*}
    x\in X, f^n(x)\to a \text { as } n\to\infty \implies f(a) =a.
\end{equation*}
\end{definition}
\begin{definition}
Let $M$ be a partially ordered monoid and $N\subset M_+$. We say that $N$ is bounded, if there exists $m\in M$ such that $n\leq m$ for all $n\in N$.  
\end{definition}
\begin{theorem}\label{th::caristiType}
Let $(X,d_{X,M},Z(M))$ be an $\mathcal{W}$-complete {$EDM$-distance} space such that the pair $(M, Z(M))$ satisfies the Weierstrass property, and $f\colon X\to X$ be an orbitally continuous operator. 
Assume that mappings $\phi \colon X\to M_+$ and $\eta\colon M_+\to M_+$ are such that for all $x\in X$
$$
\eta(d_{X,M}(x,f(x))) + \phi(f(x))\le \phi(x),
$$
$\eta$ is semi-additive (i.e. $\eta(a+b)\leq \eta(a)+\eta(b)$ for all $a,b\in M_+$), and for each  $A\subset M_+$, if $\eta(A)$  is bounded, then $A$  is bounded. Then the operator $f$ has a fixed point. 
\end{theorem}

\begin{proof}
Set $x_n = f^n(x_0)$, $n\in\NN$. We prove by induction on $n\in\NN$ that for arbitrary $k\geq 0$, one has
\begin{equation}\label{inductionClaim}
\eta(d_{X,M}(x_k, x_{k+1}) + \ldots + d_{X,M}(x_{k+n-1}, x_{k+n}))\leq \phi(x_k).
\end{equation}
If $n = 1$ and $k\geq 0$, then
$$
\eta(d_{X,M} (x_k, x_{k+1}))
  \leq 
\eta(d_{X,M}(x_k, x_{k+1})) + \phi(x_{k+1})
\leq 
\phi(x_k).
$$
Assume~\eqref{inductionClaim} holds for some $n\in\NN$ and arbitrary $k\geq 0$. Then
\begin{gather*}
\eta(d_{X,M}(x_k, x_{k+1}) +  (d_{X,M}(x_{k+1}, x_{k+2}) + \ldots + d_{X,M}(x_{k+n}, x_{k+n+ 1})))
\\ \leq 
\eta(d_{X,M}(x_k, x_{k+1})) + \eta (d_{X,M}(x_{k+1}, x_{k+2}) + \ldots + d_{X,M}(x_{k+n}, x_{k+n+ 1}))
 \\\leq 
\eta(d_{X,M}(x_k, x_{k+1})) +  \phi(x_{k+1})
\leq 
\phi(x_k),
\end{gather*}
as desired. From~\eqref{inductionClaim} we obtain that for arbitrary $n\in\NN$, $\eta\left(\sum\nolimits\nolimits_{k=1}^n d_{X,M}(x_{k-1}, x_{k})\right)$ is bounded by $\phi(x_0)$, and hence the partial sums $\sum\nolimits\nolimits_{k=1}^n d_{X,M}(x_{k-1}, x_{k})$, $n\in\NN$, are bounded. Due to the Weierstrass property, this implies that $\{x_n\}\in \mathcal{CW}$, and hence  $\{x_n\}$ converges to some point $x\in X$. Due to orbitally continuity of $f$, $x$ is a fixed point of $f$.
\end{proof}

\subsection{ Sequential contractions of \'Ciri\'c and Caccioppoli types in M-distance spaces}\label{s::sequentialContractions}

\begin{definition}
For a function $f\colon X\to X$ and $x_0\in X$, by $O(f,x_0)$ we denote the orbit
$$
O(f,x_0):=\{x_0,f(x_0), f^2(x_0),\ldots\}.
$$
For $n\in \ZZ_+$ set
$
O_n(f,x_0):=\{f^n(x_0),f^{n+1}(x_0), f^{n+2}(x_0),\ldots\},
$
where $f^0(x_0) = x_0$.
We say that an orbit $O(f,x_0)$ is bounded, if the set $\left\{d_{X,M}(f^n(x_0), f^m(x_0)),n,m\geq 0\right\}$ is bounded.
\end{definition}

For a sequence $\{ \lambda_n\}_{n\in \NN}$  of  operators $\lambda_i\colon M_+\to M_+$, $i\in\NN$, for any $n\in\NN$  we set
$$\
\prod\nolimits_{i=1}^n\lambda_i=\lambda_1\circ \lambda_2\circ\ldots\circ\lambda_n.
$$ 

\begin{theorem}\label{th::contraction1}
Let $(X,d_{X,M}, Z(M))$ be a Hausdorff $EDM$-distance space, $f\colon X\to X$ be weakly orbitally continuous, and  $\{ \lambda_n\}$ be a sequence of non-decreasing operators $\lambda_i\colon M_+\to M_+$, $i\in\NN$. Suppose that $x_0\in X$ is such that  one of the following two sets of conditions holds:
\begin{enumerate}[label = \Roman*]
    \item\label{firstSet}
    \begin{enumerate}
        \item $(X,d_{X,M}, Z(M))$ is $\mathcal{H}$-complete;
        \item  the orbit $O(f,x_0)$ is bounded; 
   \item\label{lambdaIsNullFunction}  for any $\alpha\in M_+$ one has $\left\{ \left(\prod_{i=1}^n\lambda_i\right)(\alpha)\right\}\in Z(M);
    $
    \item for each $n\in\NN$ and $x,y\in  O_n(f, x_0)$ there exist $x',y'\in O_{n-1}(f, x_0)$ such that 
    \begin{equation}\label{graduatedContraction}
         d_{X,M}(x,y)\leq \lambda_n(d_{X,M}(x',y')).
    \end{equation}
    \end{enumerate}
\item\label{secondSet}
    \begin{enumerate}
    \item $(X,d_{X,M}, Z(M))$ is $\mathcal{W}$-complete;
    \item for any $n\in \NN$,
$
    d_{X,M}(f^n(x_0),f^{n+1}(x_0))\leq \lambda_n(d_{X,M}(f^{n-1}(x_0),f^n(x_0));
$
\item\label{seriesConvergence}
$
\left\{\left(\prod_{i=1}^n\lambda_i\right)(d_{X,M}(x_0, f(x_0)))\right\}\in\mathcal{CW}.
$
\end{enumerate}
\end{enumerate}
Then $f$ has a fixed point.
The fixed point is unique provided the following set of conditions holds
\begin{enumerate}[label =\Roman*,resume]
    \item\label{uniquenessConditions} 
    \begin{enumerate}
    \item Condition~\eqref{lambdaIsNullFunction} holds;     
        \item\label{wContractionB}  For arbitrary  $x_0,y_0\in X$  and arbitrary $x\in  O_n(f, x_0)$, $y\in  O_n(f, y_0)$, there exist $x'\in O_{n-1}(f, x_0)$ and $y'\in O_{n-1}(f, y_0)$ such that inequality \eqref{graduatedContraction} holds;
        \item\label{boundednessOfOrbitPairs}  For a rbitrary  $x_0,y_0\in X$, the set $\{d_{X,M}(x,y)\colon x\in  O(f, x_0),y\in  O(f, y_0)\}$ is bounded.
    \end{enumerate}
\end{enumerate}
\end{theorem}

{
\begin{remark}
We say that the mappings considered in the first set of conditions \ref{firstSet} and in the second set of conditions \ref{secondSet} of the theorem are sequential \'Ciri\'c and  Caccioppoli type contractions respectively.
\end{remark}
}

\begin{remark}
The following contraction types are partial cases of~\eqref{graduatedContraction} (with all equal to $\lambda$ operators $\lambda_i$):
\begin{enumerate}
    \item For all $x,y\in O(f,x_0)$, $d_{X,M}(f(x),f(y))\le \lambda(d_{X,M}(x,y))$;
    \item For all  $x,y\in O(f,x_0)$ there exists 
    $$
    z\in \{d_{X,M}(x,y), d_{X,M}(x,f(x)), d_{X,M}(y,f(y)), d_{X,M}(x,f(y)), d_{X,M}(f(x),y)\}
    $$
    such that $d_{X,M}(f(x),f(y))\le \lambda(z)$.
\end{enumerate}
\end{remark}

\begin{proof}
Let the first set of conditions holds. We show that $\{x_n = f^n(x_0)\}\in \mathcal{CH}$.
 For  $n,m\in\NN$, $m\ge n$, consider  $d_{X,M}(f^n(x_0),f^m(x_0))$. From the conditions it follows that there exist $k_1,l_1\ge n-1$ such that
$
d_{X,M}(f^n(x_0),f^m(x_0))\le \lambda_n (d_{X,M}(f^{k_1}(x_0),f^{l_1}(x_0))).
$
Then there exist  $k_2,l_2\ge n-2$ such that 
$
d_{X,M}(f^{k_1}(x_0),f^{l_1}(x_0))\le \lambda_{n-1} (d_{X,M}(f^{k_2}(x_0),f^{l_2}(x_0))).
$
And so on, at the $n$-th step there exist  $k_n,l_n\ge 0$ such that 
$$
d_{X,M}(f^{k_{n-1}}(x_0),f^{l_{n-1}}(x_0))\le \lambda_1 (d_{X,M}(f^{k_n}(x_0),f^{l_n}(x_0))).
$$
Thus for the considered distance, we obtain an inequality
$$
d_{X,M}(f^n(x_0),f^m(x_0))\le \left(\prod\nolimits_{i=1}^n\lambda_i\right) (d_{X,M}(f^{k_n}(x_0),f^{l_n}(x_0)))\le \left(\prod\nolimits_{i=1}^n\lambda_i\right)(\alpha)
$$
for some $\alpha\in M_+$, which implies that $\{f^n(x_0)\}\in \mathcal{CH}$.

Let now the second set of conditions holds. Then
$$
d_{X,M}(x_n,x_{n+1})
\leq 
\lambda_n(d_{X,M}(x_{n-1},x_{n}))
 \leq \ldots\leq
\left(\prod\nolimits_{i=1}^n\lambda_i\right)(d_{X,M}(x_{0},x_{1})),
$$
and hence $\{f^n(x_0)\}\in \mathcal{CW}$. 

Therefore in both cases $\{ x_n\}$ has a limit
$
x_n=f(x_{n-1})\to \overline{x},\text{ as } n\to \infty.    
$
The equality $f(\overline{x})=\overline{x}$ holds, since $f$ is weakly orbitally continuous. 

Assume that there are two points $\overline{x},\underline{x}\in X$  such that $f(\overline{x}) =\overline{x}$ and $f(\underline{x}) = \underline{x}$.
 Then for all $n\in \NN$,
$
d_{X,M}(\overline{x},\underline{x})
=
d_{X,M}(f^n(\overline{x}),f^n(\underline{x}))$ and repeating the arguments from the proof of the first case, we obtain that $d_{X,M}(\overline{x},\underline{x})\leq \left(\prod\nolimits_{i=1}^n\lambda_i\right)(\beta),$
where $\beta\in M_+$ is some upper bound of the set $\{d_{X,M}(x,y)\colon x\in  O(f, \overline{x}),y\in  O(f, \underline{x})\}$.
Hence, due to arbitrariness of $n\in\NN$, 
 we obtain $\{d_{X,M}(\overline{x},\underline{x})\}\in Z(M)$, thus $d_{X,M}(\overline{x},\underline{x}) = \theta_M$ and $\overline{x}=\underline{x}$.
\end{proof}
The following two examples show that consideration of a sequence $\{\lambda_i\}$ of operators instead of one (i.e. the case $\lambda_i = \lambda_j$ for all $i,j\in \NN$) might be useful in some situations.

Let $M_+=\RR_+$ and $\lambda_n(t)=\frac n{n+1}t.$  Then $\left(\prod\nolimits_{i=1}^n\lambda_i\right)(t)=\frac 1{n+1}t$ and hence condition~\eqref{lambdaIsNullFunction} of Theorem~\ref{th::contraction1} is satisfied. Note that condition~\eqref{graduatedContraction} with such sequence $\{\lambda_i\}$ guarantees the inequality 
$d_{X,M}(x,y)< d_{X,M}(x',y')$, but generally speaking not the inequality 
\begin{equation}\label{realValuedContraction}
   d_{X,M}(x,y)< \alpha\cdot d_{X,M}(x',y')\text{ with some } \alpha\in (0,1). 
\end{equation}

Let $\lambda_n(t)=\left(\frac n{n+1}\right)^2t. $  Then $\left(\prod\nolimits_{i=1}^n\lambda_i\right)(t)=\left(\frac 1{n+1}\right)^2t$ and condition~\eqref{seriesConvergence} of Theorem~\ref{th::contraction1} holds, but inequality~\eqref{realValuedContraction} generally speaking does not.

\subsection{Sequential contractions in partially ordered M-distance spaces}

\begin{definition}
An $EDM$-distance space $(X,d_{X,M}, Z(M))$ is called partially ordered, if $X$ is a partially ordered set.
\end{definition}

\begin{definition}
We say that $Z(M)$-convergence in a partially ordered  $EDM$-distance space is regular, if the following two conditions are satisfied:
\begin{enumerate}
    \item For each non-decreasing sequence  $\{ x_n\}$ that converges to $x$ one has $x_n\leq x$ for all $n\in\NN$;
    \item  For each non-decreasing sequence $\{x_n\}$ such that $x_n\to x$ as $n\to\infty$, and $x_n \le y$ for all $n\in\NN$, one has $x\leq y$.
\end{enumerate}
\end{definition}
 \begin{definition}
Let $X$ be partially ordered.  We say that a function $f\colon X\to X$ is orbitally semi-continuous (with respect to the partial order $\le$ in $X$), if for each $x_0\in X$ such that $\{ f^n(x_0)\}$ converges to some $x\in X$ one has $f(x)\le x$. 
\end{definition}
 \begin{theorem}\label{th::monotoneMappingContraction}
Let $(X,d_{X,M}, Z(M))$ be a partially ordered Hausdorff $EDM$-distance space with regular $Z(M)$-convergence. Let an operator $f\colon X\to X$ be  non-decreasing and orbitally semi-continuous. Assume that there exists $x_0\in X$ such that $x_0\leq f(x_0)$ and one of the sets \eqref{firstSet} or \eqref{secondSet} of conditions from Theorem~\ref{th::contraction1} holds.
Then the operator $f$ has at least one fixed point, which is the limit of the sequence  $\{x_n = f(x_{n-1})\}$. 

The fixed point is unique, provided the set~\eqref{uniquenessConditions} of conditions from Theorem~\ref{th::contraction1} holds. If $X$ satisfies the Riesz property, then it is enough to require that conditions~\eqref{wContractionB} and~\eqref{boundednessOfOrbitPairs} hold only for orbits $O(f,x_0)$ and $O(f,y_0)$ with pairwise comparable elements.
\end{theorem}

\begin{proof}

Using the same arguments as in the proof of Theorem~\ref{th::contraction1}, we prove that in both cases, when one of the sets of conditions~\eqref{firstSet} or~\eqref{secondSet} holds, the sequence $\{ x_n = f^n(x_0)\}$ has a limit
$
x_n=f(x_{n-1})\to \overline{x},\text{ as } n\to \infty.    
$

From the conditions of the theorem, we obtain that $x_0 \leq f(x_0) = x_1$, $x_1 = f(x_0)\leq f(x_1) = x_2$, and so on. Hence $x_n\leq x_m$ for arbitrary $m\ge n\in \ZZ_+$. Due to regularity of convergence, we obtain $x_n\le \overline{x}$ for all $n\in \NN$. In virtue of monotonicity of  $f$, one has $f(x_n)\le f(\overline{x})$ for all $n\in \NN$, and hence $x_n\le f(\overline{x})$. Using regularity of convergence once again, we obtain  $\overline{x}\le f(\overline{x})$. On the other hand, using orbital semi-continuity of $f$, we obtain $f(\overline{x})\le \overline{x}$, so that $f(\overline{x})=\overline{x}$. Thus  $\overline{x}$ is a fixed point of $f$. 

If all conditions of the set~\eqref{uniquenessConditions} hold, then the proof of uniqueness of the fixed point goes as in Theorem~\ref{th::contraction1}. Otherwise (when inequality~\eqref{graduatedContraction} holds only for orbits $O(f,x_0)$ and $O(f,y_0)$ with pairwise comparable elements and $X$ satisfies the Riesz   property)  we can obtain the uniqueness as follows.

 Assume that there are two fixed points $\overline{x},\underline{x}$ of the mapping $f$. Consider 
$d_{X,M}(\overline{x}, f^n(\underline{x}\vee \overline{x}))=d_{X,M}(f^n(\overline{x}), f^n(\underline{x}\vee \overline{x}))$. Using the arguments form the proof of Theorem~\ref{th::contraction1} and taking into consideration the fact that the elements of the orbits $O(f,\overline{x})$ (in which all elements are equal to $\overline{x}$) and $O(f,\underline{x}\vee\overline{x})$ are pairwise comparable, we obtain {for each $n\in\NN$}
$$
d_{X,M}(\overline{x}, f^n(\underline{x}\vee \overline{x}))=d_{X,M}(f^n(\overline{x}),f^n(\underline{x}\vee \overline{x}))
\le  \left(\prod\nolimits_{i=1}^n\lambda_i\right)(\beta)
$$
with some $\beta\in M_+$ {independent of $n$}. Hence $\{ d_{X,M}(\overline{x}, f^n(\underline{x}\vee \overline{x}))\}\in Z(M)$. Analogously $\{ d_{X,M}(\underline{x}, f^n(\underline{x}\vee \overline{x}))\}\in Z(M)$.  Since the space $(X,d_{X,M},Z(M))$ is Hausdorff, we obtain $\underline{x}=  \overline{x}$.
\end{proof}

\section{Applications }\label{s::applications}
{
Theorem~\ref{th::M-KContraction} can be applied to different $EDM_\zeta$-metric spaces, and
Theorems~\ref{th::caristiType}, \ref{th::contraction1} and~\ref{th::monotoneMappingContraction} can be applied to different Hausdorff distance spaces, in particular to the ones discussed in Section~\ref{s::examples}. 
It seems that even in the case $M = \RR_+$ and these examples of distance spaces  some of these results will be new. Below we discuss some of such consequences in more details.
}

\subsection{The Fredholm integral equations}\label{s::fredholmEquation}
Let $T$ be a metric compact and $\mu$ be a measure defined on the $\sigma$-algebra of Borel sets of $T$. Let also {an} $L$-space {(the definition can be found in~\cite{Vahrameev, Babenko19, Babenko20,Babenko21})}  $(X,h_X)$ be given, and $C(T,X)$ be the space of continuous functions $x\colon T\to X$. In the space $C(T,X)$ consider a $C(T,\mathbb{R})$-valued metric, setting
$
h_{C(T,X)}(x,y)=h_X(x(\cdot),y(\cdot)),
$
where $C(T,\RR)$ is considered as a partially ordered monoid with pointwise addition and partial order.
It is easy to see that the obtained space $C(T,X)$ is complete whenever $X$ is complete.

We are interested, whether a solution of the equation
$$
x(t)=f(t)+\int_Tg(t,s,x(s))d\mu(s)
$$
exists and is unique, where
$g\in C(T\times T\times X, X)$ and $f\in C(T,X)$ are given and $x$ is to be found.
Assume that for arbitrary $x,y\in C(T,X)$ and $t,s\in T$, 
$$
h_{X}(g(t,s,x(s)),g(t,s,y(s)))\le Q(t,s)h_X(x(s),y(s)),
$$
where $Q\in C(T\times T,\mathbb{R}_+)$. If, for example, $g(t,s,x(s))=K(t,s)x(s)$, where $K\in C(T\times T,\mathbb{R}_+)$, i.e. the considered equation is linear with non-negative kernel, then $Q(t,s)=K(t,s)$. 

For the operator $A\colon C(T,X)\to C(T,X)$
$$
Ax(\cdot)=f(\cdot)+\int_Tg(\cdot,s,x(s))d\mu(s)
$$
we have
$$
h_{C(T,X)}(Ax,Ay)\le \lambda(h_{C(T,X)}(x,y)),
$$
where $\lambda\colon C(T,\RR_+)\to C(T,\RR_+)$ is the linear integral operator with the kernel $Q(t,s)$:
$$
\lambda(x(\cdot))(t)=\int_TQ(t,s)x(s)d\mu(s).
$$
Setting $Q_1(t,s)=Q(t,s)$ and $Q_n(t,s)=\int_TQ_{n-1}(t,u)Q_1(u,s)d\mu(u)$ for $n\in\mathbb{N}$, we obtain
$$
\lambda^n(x(\cdot))(t)=\int_TQ_n(t,s)x(s)d\mu(s).
$$
Theorem~\ref{th::contraction1} (conditions~\eqref{secondSet} with all equal $\lambda_i$) is applicable to the operator $A$ if for each function $x\in C(T,X_+)$ the series
$
\sum\nolimits\nolimits_{n=1}^\infty \int_TQ_n(t,s)x(s)d\mu(s)
$
converges in  $(C(T,X),h_{C(T,X)})$.
It is easy to see that it is sufficient to require that the series
$
\sum\nolimits\nolimits_{n=1}^\infty \int_TQ_n(t,s)d\mu(s)
$
converges in the space $(C(T,X),h_{C(T,X)})$. Thus, if the latter series converges, then the considered Fredholm equation has a unique solution in the space $C(T,X)$. It is easy to see that this condition is significantly more general than the one from~\cite{Babenko19}.

 \subsection{Multiple fixed point theorems}
 We use a generalization of an approach to the multiple fixed point theorems suggested in~\cite{ChobanBerinde}.
Let $(Y,d_{Y,M})$ be a partially ordered $M$-distance space, $A$ be some non-empty set of indices and $P\colon A\to  \{0,1\}$ be given. Define a partial order in the Cartesian product  $Y^A$, setting for $x, y\in Y^A$ (i.e. $x,y\colon A\to Y$)
\begin{multline*}
x\preccurlyeq_{P} y \text{ if and only if }  x(\alpha)\le y(\alpha) \text{ for } \alpha\in A \text{ such that } P(\alpha) = 0, 
\\
\text{ and } y(\alpha)\le x(\alpha)\text{ for other } \alpha\in A.
\end{multline*} 
We say that a mapping $f\colon Y^A\to Y^A$ is $P$-non-decreasing, if
$
x\preccurlyeq_{P} y\implies f(x)\preccurlyeq_{P} f(y).
$

Let mappings
$\sigma\colon A^2\to A$
 and $f\colon Y^A\to Y$  be given. These mappings generate the mapping
$\sigma f\colon Y^A\to Y^A$
in the following way:
$$\sigma f(x)=\alpha\mapsto f(x\circ \sigma(\alpha, \cdot)), x\in Y^A.$$

\begin{definition}
An element $x\in Y^A$ is called a $\sigma$-multiple fixed point of the mapping $f\colon Y^A\to Y$, if $x= \sigma f(x)$.
\end{definition}

Let $Z(M)$ be a family of null sequences in the partially ordered monoid $M$. As in the last case of Example~\ref{ex::CartesianProducts} from Section~\ref{s::examples} in the space $Y^A$ we introduce an $M^A$-valued distance $d_{Y^A, M^A}$ and a family $Z(M^A)$ of null sequences in $M^A$.
With the partial order $\preccurlyeq_P$ and such distance function we obtain a partially ordered distance space $(Y^A, d_{Y^A,M^A},Z(M^A))$.

 Let $(Y,d_{Y,M}, Z(M))$ be a partially ordered distance space with partial order $\leq$. We say that $Z(M)$-convergence in it is co-regular, if it is regular with respect to the opposite partial order $\geq$.

 \begin{theorem}
Let $(Y,d_{Y,M}, Z(M))$ be a partially ordered $\mathcal{W}$-complete Hausdorff $EDM$-distance space with  both regular and co-regular $Z(M)$-convergence. Let $A$ be some non-empty set of indices, $P\colon A\to \{0,1\}$ and $\sigma\colon A^2\to A$ be some functions, and an operator $g\colon Y^A\to Y$ be such that $\sigma g$ is $P$-non-decreasing and orbitally semi-continuous. Assume that there exists $x^0\in Y^A$ such that $x^0\preccurlyeq_P \sigma g(x^0)$ and $\lambda\colon M^A\to M^A$ is $P$-non-decreasing and such that 
$$
x\preccurlyeq_{P} y\implies d_{Y^A, M^A}(\sigma g(x), \sigma g(y)) \leq \lambda (d_{Y^A, M^A}(x,y)),
$$
and
$$
\{\lambda^n(d_{Y^A, M^A}(x^0, \sigma g(x^0))\}\in \mathcal{CW}.
$$
Then the operator $g$ has a $\sigma$-multiple fixed point, which is the limit of the sequence  $\{x^n = \sigma g(x^{n-1})\}$.
\end{theorem}
\begin{proof}
Due to the coordinate-wise nature of the convergence in the space $(Y^A, d_{Y^A, M^A}, Z(M^A))$, it is $\mathcal{W}$-complete (Hausdorff), provided the space $(Y, d_{Y, M},Z(M))$ is $\mathcal{W}$-complete (resp. Hausdorff). Moreover, due to the definition of the partial order $\preccurlyeq_P$, the $Z(M^A)$ convergence in $(Y^A, d_{Y^A, M^A}, Z(M^A))$ is regular  (and co-regular), provided the $Z(M)$-convergence in $(Y,d_{Y,M},Z(M))$ is regular and co-regular. Hence we can apply Theorem~\ref{th::monotoneMappingContraction} (its case when conditions~\ref{secondSet} hold) to the partially ordered distance space $(X,d_{X,M},Z(M)) = (Y^A, d_{Y^A, M^A}, Z(M^A))$ and $f = \sigma g$. \end{proof}
\subsection{Random fixed point theorems}
Let $(X, d_{X,M}, Z(M))$ be a Hausdorff $EDM$-distance space and  $\Omega$ be some non-empty set. 
Let the point-wise convergence in  $X^\Omega$ be defined and some class of functions $\mathfrak{M}\subset X^\Omega$ that is closed with respect to the point-wise convergence be given.

We consider a mapping $F\colon \Omega\times X\to X$ that satisfies the following properties:
\begin{enumerate}[label={\arabic*)}]
    \item For all $ x\in X$, $F(\cdot,x)\in \mathfrak{M}$;
    \item For all  $\omega\in \Omega$ there exists $x_0(\omega)\in X$ such that the sequence of iterations
    $\{ x_n(\omega)\}$,
    $$
    x_1(\omega)=F(\omega,x_0(\omega)),x_2(\omega)=F(\omega,x_1(\omega)),\ldots,x_n(\omega)=F(\omega,x_{n-1}(\omega)),\ldots,
    $$
    belongs to  $\mathfrak{M}$ and converges to a fixed point of the mapping $F(\omega,\cdot)$ .
\end{enumerate}
Then there exists a function $\overline{x}(\cdot)\in\mathfrak{M} $ such that for all $\omega\in \Omega$,
\begin{equation}\label{randomFixedPoint}
\overline{x}(\omega)=F(\omega,\overline{x}(\omega)).
\end{equation}
Condition 2) holds for example in the case, when for all  $\omega\in \Omega$ the function $F(\omega,\cdot)$ satisfies the conditions of one of the theorems from Section~\ref{s::contractionTheorems}.

The function that satisfies~\eqref{randomFixedPoint} can be called a $\mathfrak{M}$-fixed point of $F$. In the case, when $M=\RR_+$, some measurable space $(\Omega, U)$ is given and $\mathfrak{M}$ is the set of measurable functions, then we obtain a result regarding random fixed point theorems. See e.g.~\cite{Nietoatall} and references therein.

\bibliographystyle{plain}
\bibliography{bibliography}

\end{document}